\documentclass[12pt,letterpaper]{article}

\usepackage{amsmath}
\usepackage{amssymb}
\usepackage{latexsym} 
\usepackage{array}
\usepackage{enumerate}
\usepackage{amsthm}

\newtheorem{theorem}{Theorem}

\begin{document}

\title{\bf Correctness of depiction in planar diagrams of spatial figures}
\author{P. L. Robinson}
\date{}
\maketitle

\bigbreak 

This study was motivated in part by the following question: is it possible to decide whether a given planar diagram correctly depicts a given spatial figure? We do not propose to address this question in full generality or even to define exactly what it means. Instead, we shall precisely formulate a special case for which we offer a complete answer. To be specific, we pose the following \\

\noindent
{\bf Question}: Let $\pi$ be a plane in which $P_1 Q_1 R_1 S_1$ and $P_2 Q_2 R_2 S_2$ are quadrangles; assume that $P_1 P_2$, $Q_1 Q_2$, $R_1 R_2$, $S_1 S_2$ all pass through the point $O$. Is it possible to decide whether this diagram is a correct two-dimensional depiction of the three-dimensional figure comprising a quadrangle `$P_1 Q_1 R_1 S_1$' in one plane along with its quadrangular shadow `$P_2 Q_2 R_2 S_2$' in another plane as projected from a `light source' at `$O$'? \\

We shall show that it is indeed possible to make this decision, by checking a simple condition whose necessity and sufficiency follow from Theorems 1 and 2 respectively. \\

Naturally, we view this problem as belonging to the province of Projective Geometry; in particular, we accept that any two lines in the same plane have a point of intersection. We defer to the authority of Veblen and Young [2] for a classic treatment of the subject; Chapters I and II more than cover most of what we require. For a more recent account, see Coxeter [1]. \\

It will be convenient to fix some notation, to be used throughout. Let $P Q R S$ be a plane quadrangle, no three of whose vertices $P, Q, R, S$ are collinear. Its diagonal triangle $ABC$ has as vertices the points in which its opposite sides intersect: 
	\[A = SP \cdot QR, \; \; B = SQ \cdot RP, \; \; C = SR \cdot PQ. 
\]
When the vertices $P, Q, R, S$ are decorated with overlines or subscripts, the diagonal points $A, B, C$ will be decorated correspondingly. Perspectivity has the customary meaning: for instance, when the lines $A_1 A_2$, $B_1 B_2$, $C_1 C_2$ all pass through $O$ we say that the triangles $A_1 B_1 C_1$ and $A_2 B_2 C_2$ are perspective from $O$ and write 
	\[A_1 B_1 C_1 \stackrel{O}\doublebarwedge A_2 B_2 C_2 .
\]

\begin{theorem} 
Let $O$ be a point not on either of the distinct planes $\overline{\pi}$ and $\pi$. Let $\overline{P} \overline{Q} \overline{R} \overline{S}$ be a quadrangle in $\overline{\pi}$ and let $P Q R S$ be a quadrangle in $\pi$. If $\overline{P} \overline{Q} \overline{R} \overline{S} \stackrel{O}\doublebarwedge P Q R S$ then $\overline{A} \overline{B} \overline{C} \stackrel{O}\doublebarwedge A B C$.
\end{theorem} 

\begin{proof} 
The lines $\overline{S} S$ and $\overline{Q} Q$ meet in $O$, so the points $O, S, \overline{S}, Q, \overline{Q}$ lie on a plane, which contains the point $\overline{B}$ of $\overline{S} \overline{Q}$. Accordingly, the (coplanar) lines $O \overline{B}$ and $S Q$ meet; similarly, $O \overline{B}$ and $R P$ meet. Consequently, $O \overline{B}$ meets the plane $\pi$ in the point 
	\[O \overline{B} \cdot \pi = SQ \cdot RP = B. 
\]
In like manner, $O \overline{C} \cdot \pi = C$ and $O \overline{A} \cdot \pi = A$. We preferred to focus on $B$ since this diagonal point lies `inside' the quadrangle $P Q R S$ when it is drawn in the `obvious' way.
\end{proof} 

Thus, perspectivity of two {\it simple} quadrangles in different planes forces perspectivity of the {\it complete} quadrangles. \\

In terms of our original {\bf Question}, the preceding theorem yields perspectivity from $O$ of the diagonal triangles as a necessary condition for correctness of depiction. In order to establish that this condition is also sufficient, we must attend to obligatory special cases as usual; we prefer to frame this attention as a preparatory discussion, rather than as a formal theorem. \\

Thus, let $\pi$ be a plane in which the quadrangles $P_1 Q_1 R_1 S_1$ and $P_2 Q_2 R_2 S_2$ are perspective from a point $O$. Suppose that the quadrangles are so placed that at least one side in each opposite pair equals its homologue under the perspectivity. There are two cases to consider: 
($\Delta$) three sides that equal their homologues make up a triangle; 
($\bullet$) three sides that equal their homologues meet at a vertex. \par
Case ($\Delta$): Say $Q_1 R_1 = Q_2 R_2$, $R_1 P_1 = R_2 P_2$, $P_1 Q_1 = P_2 Q_2$. In this case, $P_1 = P_2$, $Q_1 = Q_2$, $R_1 = R_2$: in fact, if $R_1 \neq R_2$ then as $R_1$ lies on $Q_1 R_1$ and $R_2$ lies on $Q_2 R_2$ it follows that $Q_1 R_1 = Q_2 R_2 = R_1 R_2$ while $R_1 R_2 = R_1 P_1 = R_2 P_2$ follows similarly; but now $R_1 R_2 = P_1 Q_1 = P_2 Q_2$ violates non-collinearity of $P_1, Q_1, R_1$ and $P_2, Q_2, R_2$. Thus: the quadrangles have the form $P Q R S_1$ and $P Q R S_2$. \par
Case ($\bullet$): Say $S_1 P_1 = S_2 P_2$, $S_1 Q_1 = S_2 Q_2$, $S_1 R_1 = S_2 R_2$. In this case, $S_1 = S_2$: in fact, if $S_1 \neq S_2$ then $S_1 S_2 = S_1 R_1 = S_2 R_2$ and so on, whence $S_1 S_2$ passes through the non-collinear points $P_1, Q_1, R_1$ and $P_2, Q_2, R_2$. The quadrangles now have the form $P_1 Q_1 R_1 S$ and $P_2 Q_2 R_2 S$. We claim that among the pairs $(P_1, P_2), (Q_1, Q_2), (R_1, R_2)$ at most one can have distinct entries: indeed, if $P_1 \neq P_2$ and $Q_1 \neq Q_2$ then $O S = P_1 P_2 = Q_1 Q_2$ in violation of non-collinearity. Thus: the quadrangles have the form $P Q R_1 S$ and $P Q R_2 S$. \\
 
The conclusion to this preparatory discussion is that if at least one side in each opposite pair agrees with its homologue then at most one homologous pair of vertices is distinct.

\begin{theorem} 
In a plane $\pi$, let the quadrangles $P_1 Q_1 R_1 S_1$ and $P_2 Q_2 R_2 S_2$ be perspective from $O$. If their diagonal triangles $A_1 B_1 C_1$ and $A_2 B_2 C_2$ are also perspective from $O$ then there exists a quadrangle $\overline{P} \overline{Q} \overline{R} \overline{S}$ in a plane $\overline{\pi} \neq \pi$ along with points $O_1 \neq O_2$ not in either plane, such that 
	\[P_1 Q_1 R_1 S_1 \stackrel{O_1}\doublebarwedge \overline{P} \overline{Q} \overline{R} \overline{S} \stackrel{O_2}\doublebarwedge P_2 Q_2 R_2 S_2. 
\]
\end{theorem} 

\begin{proof} 
Let $O_1$ and $O_2$ be distinct points collinear with $O$ but not in $\pi$. The lines $O O_1 O_2$ and $O S_1 S_2$ meet in $O$; thus $O, O_1, O_2, S_1, S_2$ are coplanar, so the lines $O_1 S_1$ and $O_2 S_2$ meet, say in the point $\overline{S} = O_1 S_1 \cdot O_2 S_2$. Define the points $\overline{P}, \overline{Q}, \overline{R}; \overline{A}, \overline{B}, \overline{C}$ analogously. No three of $\overline{P}, \overline{Q}, \overline{R}, \overline{S}$ are collinear: if $\overline{P}, \overline{Q}, \overline{R}$ were collinear, then the plane through $O_1 \overline{P} \overline{Q} \overline{R}$ would meet the (distinct) plane $\pi$ in a line containing $P_1, Q_1, R_1$ and so render these points collinear. All that remains is to see that the quadrangle $\overline{P} \overline{Q} \overline{R} \overline{S}$ lies in a plane $\overline{\pi}$ (necessarily distinct from $\pi$). Suppose that each side of some opposite pair in $P_1 Q_1 R_1 S_1$ is distinct from its homologue in $P_2 Q_2 R_2 S_2$: say $P_1 Q_1 \neq P_2 Q_2$ and $R_1 S_1 \neq R_2 S_2$. The planes $\pi_1 = O_1 P_1 Q_1$ and $\pi_2 = O_2 P_2 Q_2$ are distinct, so their intersection $\pi_1 \cdot \pi_2$ is a line. Now $C_1$ lies on $P_1 Q_1$ and $C_2$ lies on $P_2 Q_2$ so that $\overline{C} = O_1 C_1 \cdot O_2 C_2$ lies on $O_1 P_1 Q_1 \cdot O_2 P_2 Q_2 = \pi_1 \cdot \pi_2$; this line contains $\overline{P}$ and $\overline{Q}$ likewise. Accordingly, $\overline{P}, \overline{Q}, \overline{C}$ are collinear; similarly, $\overline{C}, \overline{R}, \overline{S}$ are collinear. Thus $\overline{P} \overline{Q}$ meets $\overline{R} \overline{S}$ (in $\overline{C}$) and so $\overline{P} \overline{Q} \overline{R} \overline{S}$ is indeed planar. \\

In the complementary case that at least one side in each opposite pair agrees with its homologue, the preparatory discussion prior to the theorem shows that we may take the quadrangles to have the form $P Q R S_1$ and $P Q R S_2$. In this case, if $A_1 \neq A_2$ and $B_1 \neq B_2$ then $Q R = A_1 A_2$ and $R P = B_1 B_2$ both pass through $O$; this places the non-collinear points $P, Q, R$ on the same line through $O$. It follows that among $(A_1, A_2), (B_1, B_2), (C_1, C_2)$ at least two pairs have entries that agree; say $A_1 = A_2 = A$ and $B_1 = B_2$. Now $A = S_1 P \cdot Q R = S_2 P \cdot Q R$ implies that $A P = S_1 S_2$ whence $S_1 S_2$ passes through $P$; likewise, $S_1 S_2$ passes through $Q$. The resulting equality $S_1 S_2 = P Q$ contradicts non-collinearity one last time and shows that this complementary case does not arise. 
\end{proof} 

Observe that the complete quadrangles $P_1 Q_1 R_1 S_1$ and $P_2 Q_2 R_2 S_2$ thus correspond under a perspective collineation, with centre $O$ and axis $\overline{\pi} \cdot \pi$. \\

We were careful to offer a proof of Theorem 2 in full generality, making no special assumptions on the placement of the two quadrangles other than those declared in the statement of the theorem. Of course, if such simplifying assumptions are made, simplified proofs are possible. For example, if we assume that the two quadrangles have distinct homologous sides, then the proof of Theorem 2 offered above goes through without the need to consider the complementary case. If we assume instead that the two quadrangles have distinct homologous vertices then again the proof of Theorem 2 goes through without the complementary case (which involves coincident homologous vertices).  \\

An alternative approach to Theorem 2 is of independent interest, so we offer it here. As our original approach was completely general, we shall feel free to make a simplifying assumption of general position (announced in italics below) and leave consideration of the complementary case and incident issues as an exercise for the reader. We use the theorem of Desargues and its converse, pertaining to perspective triangles: see [2] Chapter II Theorems 1 and 1$'$; also [2] Chapter 2 Theorems 2.32 and 2.31. \\

As the triangles $P_1 Q_1 R_1$ and $P_2 Q_2 R_2$ are perspective from a point (namely, $O$) they are (Desargues) perspective from a line: that is, the pairwise intersections $Q_1 R_1 \cdot Q_2 R_2$, $R_1 P_1 \cdot R_2 P_2$, $P_1 Q_1 \cdot P_2 Q_2$ of homologous sides lie on a line, say $s$. Dropping instead the points $R, Q, P$ from the quadrangles leads similarly to lines $r, q, p$ on which lie intersections as follows: \\

($s$) \; \; $Q_1 R_1 \cdot Q_2 R_2$ \; \; $R_1 P_1 \cdot R_2 P_2$ \; \; $P_1 Q_1 \cdot P_2 Q_2$ \\

($r$) \; \; $P_1 Q_1 \cdot P_2 Q_2$ \; \; $Q_1 S_1 \cdot Q_2 S_2$ \; \; $S_1 P_1 \cdot S_2 P_2$ \\

($q$) \; \; $S_1 P_1 \cdot S_2 P_2$ \; \; $P_1 R_1 \cdot P_2 R_2$ \; \; $R_1 S_1 \cdot R_2 S_2$ \\

($p$) \; \; $R_1 S_1 \cdot R_2 S_2$ \; \; $S_1 Q_1 \cdot S_2 Q_2$ \; \; $Q_1 R_1 \cdot Q_2 R_2$. \\

{\it Simplifying assumption}: six different pairwise intersection points are displayed here. That the intersections be points is of course equivalent to distinctness of homologous sides; that all six be different is then equivalent to distinctness of homologous vertices (non-collinearity again). \\

Now, suppose that $C_1 C_2$ passes through $O$. The triangles $C_1 S_1 P_1$ and $C_2 S_2 P_2$ are perspective from $O$ so (Desargues) the pairwise intersections $S_1 P_1 \cdot S_2 P_2$, $P_1 C_1 \cdot P_2 C_2 = P_1 Q_1 \cdot P_2 Q_2$, $C_1 S_1 \cdot C_2 S_2 = R_1 S_1 \cdot R_2 S_2$ lie on a line; this line shares two points with $r$ and two points with $q$ whence $r = q$. Similarly, perspectivity of $C_1 S_1 Q_1$ and $C_2 S_2 Q_2$ yields $r = p$. All three of the intersections on $s$ now lie on $r = q = p$: namely, $Q_1 R_1 \cdot Q_2 R_2$ on $p$, $R_1 P_1 \cdot R_2 P_2$ on $q$, $P_1 Q_1 \cdot P_2 Q_2$ on $r$. It follows that all four lines coincide: $s = r = q = p =: o$, say. \\

Next, consider the triangles $A_1 P_1 Q_1$ and $A_2 P_2 Q_2$: the pairwise intersections of their homologous sides all lie on the line $o$, so (Desargues, converse) the lines $A_1 A_2$, $P_1 P_2$ and $Q_1 Q_2$ are concurrent; perspectivity of $B_1 P_1 Q_1$ and $B_2 P_2 Q_2$ likewise passes $B_1 B_2$, $P_1 P_2$ and $Q_1 Q_2$ through a point. The point of concurrence is $O$: the possibility $P_1 P_2 = Q_1 Q_2$ may be sidestepped by considering also the triangles with vertices $A P R$ and $B Q R$, for it cannot be  (non-collinearity!) that $P_1 P_2 = Q_1 Q_2 = R_1 R_2$. It follows that $A_1 A_2$ and $B_1 B_2$ also pass through $O$. \\

Construction of a quadrangle $\overline{P} \overline{Q} \overline{R} \overline{S}$ of which $P_1 Q_1 R_1 S_1$ and $P_2 Q_2 R_2 S_2$ are shadows may now proceed somewhat differently, as follows. Choose any plane $\overline{\pi} \neq \pi$ through the line $o$ and choose a point $O_1$ not on either plane. Define $\overline{S} = \overline{\pi} \cdot O_1 S_1$ and define $\overline{P}, \overline{Q}, \overline{R}$ analogously. Planarity of the quadrangle $\overline{P} \overline{Q} \overline{R} \overline{S}$ is plain. As $\overline{S}$ lies on $O_1 S_1$ and $\overline{R}$ lies on $O_1 R_1$, the lines $\overline{S} \overline{R}$ and $S_1 R_1$ meet (on $\overline{\pi} \cdot \pi = o$) necessarily at $S_1 R_1 \cdot o = S_1 R_1 \cdot S_2 R_2$. Concurrence of $\overline{P} \overline{Q}$, $P_1 Q_1$ and $P_2 Q_2$ (and so on) is shown in the same way. As the points $\overline{Q} \overline{R} \cdot Q_2 R_2$, $\overline{R} \overline{P} \cdot R_2 P_2$, $\overline{P} \overline{Q} \cdot P_2 Q_2$ all lie on $o$ it follows (Desargues, converse) that the lines $\overline{P} P_2$, $\overline{Q} Q_2$, $\overline{R} R_2$ all pass through one point, $O_2$ say; $\overline{S} S_2$ clearly passes through the same point. Finally, observe that concurrence of the lines $\overline{S} \overline{R}$, $S_1 R_1$ and $S_2 R_2$ implies (Desargues) collinearity of $O = S_1 S_2 \cdot R_1 R_2$, $O_1 = \overline{S} S_1 \cdot \overline{R} R_1$ and $O_2 = \overline{S} S_2 \cdot \overline{R} R_2$. \\
 
This completes an alternative approach to Theorem 2. Note the additional finding that correctness of depiction may be verified by testing just one homologous pair of diagonal points: if $O$ is collinear with one homologous pair, then $O$ is collinear with each. Note also the finding that the lines $p, q, r, s$ coincide; this common line $o$ meets the sides of $P_1 Q_1 R_1 S_1$ and $P_2 Q_2 R_2 S_2$ (and $\overline{P} \overline{Q} \overline{R} \overline{S}$ indeed) in the points of one and the same quadrangular set. This has a bearing on [2] page 51 exercise 2 and [1] page 22 exercise 2: there it was shown that if two (similarly placed) quadrangles determine the same quadrangular set then their diagonal triangles are perspective; here we have a converse. \\

We close by remarking that our criterion for correctness of depiction (namely, that the diagonal triangles also be perspective) is eminently reasonable on `physical' grounds: if the spatial quadrangle represented by $P_1 Q_1 R_1 S_1$ is planar, then its diagonals $S_1 Q_1$ and $R_1 P_1$ meet in a material point $B_1$ having $B_2$ as shadow; if it is not planar, then the `intersection' $B_1$ is not material and the shadow `intersection' $B_2$ is a trick of the light. 

\bigbreak
\noindent 
{\bf References} 
\bigbreak 
\noindent 
[1] H.S.M. Coxeter, {\it Projective Geometry}, Second Edition, Springer-Verlag (1987). \\
\noindent
[2] O. Veblen and J.W. Young, {\it Projective Geometry}, Volume I, Ginn and Company (1910).

\bigbreak

\newpage 
\noindent 
P. L. Robinson \\
Department of Mathematics \\
University of Florida \\
Gainesville \\
FL 32611

\end{document}